\documentclass[12pt]{article}
\usepackage{amssymb}
\usepackage{amsmath, amsthm}
\usepackage{graphicx, url}
\usepackage[algo2e,ruled,vlined]{algorithm2e}
\usepackage{enumerate}
\usepackage[mathlines]{lineno}
\usepackage{color}

\usepackage{tikz}
\tikzset{every picture/.style={thick}}
\tikzset{every node/.style={circle,draw,inner sep=2pt}}

\newtheorem{thm}{Theorem}[section]
\newtheorem{cor}[thm]{Corollary}
\newtheorem{lem}[thm]{Lemma}
\newtheorem{prop}[thm]{Proposition}

\newtheorem{obs}[thm]{Observation}

\newtheorem{Model}[thm]{Model}

\theoremstyle{definition}
\newtheorem{rem}[thm]{Remark}

\theoremstyle{definition}
\newtheorem{defn}[thm]{Definition}
\newtheorem{ex}[thm]{Example}

\numberwithin{figure}{section}

\newcommand{\pd}{\gamma_P}

\newcommand{\noi}{\noindent}
\newcommand{\bit}{\begin{itemize}}
\newcommand{\eit}{\end{itemize}}
\newcommand{\ben}{\begin{enumerate}}
\newcommand{\een}{\end{enumerate}}
\newcommand{\beq}{\begin{equation}}
\newcommand{\eeq}{\end{equation}}
\newcommand{\bea}{\begin{eqnarray*}}
\newcommand{\eea}{\end{eqnarray*}}
\newcommand{\bean}{\begin{eqnarray}}
\newcommand{\eean}{\end{eqnarray}}
\newcommand{\bpf}{\begin{proof}}
\newcommand{\epf}{\end{proof}\ms}
\newcommand{\bmt}{\begin{bmatrix}}
\newcommand{\emt}{\end{bmatrix}}
\newcommand{\ms}{\medskip}

\newcommand{\beqa}{\begin{array}}
\newcommand{\eeqa}{\end{array}}

\newcommand{\lc}{\left\lceil}
\newcommand{\rc}{\right\rceil}

\begin{document}

\title{Restricted power domination and zero forcing problems}

\author{Chassidy Bozeman\thanks{Department of Mathematics, Iowa State University, Ames, IA 50011, USA (cbozeman@iastate.edu).}\and 
Boris Brimkov\thanks{Department of Computational and Applied Mathematics, Rice University, Houston, TX, 77005, USA (boris.brimkov@rice.edu)} \and 
Craig Erickson\thanks{Department of Mathematics and Computer Science, Grand View University, Des Moines, IA, USA (cerickson@grandview.edu)} \and  
Daniela Ferrero\thanks{Department of Mathematics, Texas State University, San Marcos, TX 78666, USA (dferrero@txstate.edu).} \and 
Mary Flagg\thanks{Department of Mathematics, Computer Science and Cooperative Engineering, University of St. Thomas, 3800 Montrose, Houston, TX 77006, USA (flaggm@stthom.edu).} \and 
Leslie Hogben\thanks{Department of Mathematics, Iowa State University,
Ames, IA 50011, USA  and American Institute of Mathematics, 600 E. Brokaw Road, San Jose, CA 95112, USA
(hogben@aimath.org).}}

\maketitle

\begin{abstract}
\noindent Power domination in graphs arises from the problem of monitoring an electric power system by placing as few measurement devices in the system as possible. A power dominating set of a graph is a set of vertices that observes every vertex in the graph, following a set of rules for power system monitoring. A practical problem of interest is to determine the minimum number of additional measurement devices needed to monitor a power network when the network is expanded and the existing devices remain in place. In this paper, we study the problem of finding the smallest power dominating set that contains a given set of vertices $X$. We also study the related problem of finding the smallest zero forcing set that contains a given set of vertices $X$. The sizes of such sets in a graph $G$ are respectively called the \emph{restricted power domination number} and \emph{restricted zero forcing number} of $G$ subject to $X$. We derive several tight bounds on the restricted power domination and zero forcing numbers of graphs, and relate them to other graph parameters. We also present exact and algorithmic results for computing the restricted power domination number, including integer programs for general graphs and a linear time algorithm for graphs with bounded treewidth. We also use restricted power domination to obtain a parallel algorithm for finding minimum power dominating sets in trees.
\end{abstract}

\noi {\bf Keywords:} Power domination, restricted power domination, zero forcing, restricted zero forcing  
\smallskip

\noi{\bf AMS subject classification} 05C69, 05C50, 05C57, 94C15 

\section{Introduction}

The study of the power domination number of a graph arose from the question of how to monitor electric power networks at minimum cost, and was recast in graph theoretical terms by Haynes et al.~in  \cite{HHHH}. In this model, vertices represent electric nodes and edges represent connections via transmission lines. Electric power companies need to continuously monitor their networks, and one method to do this is to place Phase Measurement Units (PMUs) at selected locations in the system. Because of the cost of a PMU, it is important to minimize the number of PMUs used while maintaining the ability to observe the entire system. The physical laws by which PMUs can observe the network give rise to propagation rules governing power domination; see Section \ref{sPDZdef} for formal definitions.

Electric networks are frequently modified, sometimes by building an extension with new nodes and lines.  The labor to install a PMU and the (non-portable) infrastructure that must be put at the node to support the PMU are significant parts of the total cost of installing a PMU at a given node.  Thus it seems useful to determine the minimum number of additional PMUs needed (and the locations where they should be placed)  when an existing network is expanded and the existing PMUs remain in place.
In this paper we consider the problem of finding a power dominating set that contains a given set of vertices $X$ (locations of the existing PMUs) and minimizes the number of additional PMUs needed, or equivalently, that minimizes the total number of PMUs used subject to the constraint that  the vertices in $X$ are included in the solution.

 The process of zero forcing  was introduced independently in combinatorial matrix theory (where it refers to forcing zero entries in a null vector of a matrix) \cite{AIM08}, and in mathematical physics in the control of quantum systems \cite{graphinfect}.  Zero forcing is closely related to power domination, because power domination can be described as a domination step followed by the zero forcing process (or, since power domination came first, zero forcing can be described as power domination without the domination step).  The  problem of finding a  zero forcing set that is minimum subject to containing a set $X$ of vertices  is also considered here.  Zero forcing sets containing specific vertices have been used in \cite{BFH17} to show the correctness of the Wavefront algorithm for zero forcing (which is implemented in \cite{mrsage}), and  explicitly or implicitly in certain other zero forcing proofs \cite{BFH17a, R12}.  This concept may also have application to control of quantum systems where certain controls are automatically available.  

This paper is organized as follows. In the next section, we recall some graph theoretic notions, specifically those related to power domination and zero forcing, and define restricted power domination and zero forcing.  Additionally, we show that our study of restricted power domination falls under the general study of restricted domination parameters initiated by Goddard and Henning \cite{GH07} as a generalization of the work by Sanchis \cite{S97}. 
In Section \ref{section_struct}, we present structural results and bounds for restricted power domination and restricted zero forcing. In Section \ref{sexactPD}, we give exact results and algorithms for the restricted power domination and restricted zero forcing numbers of certain families of graphs. We conclude with some final remarks and open questions in Section \ref{sconc}. 


 \section{Preliminaries}\label{sPDZdef}


A  {\em graph} $G=(V,E)$  is an ordered pair consisting of a finite nonempty set of {\em  vertices} $V=V(G)$ and a set of {\em edges} $E=E(G)$ containing unordered pairs of distinct vertices (all graphs discussed in this paper are simple, undirected, and finite). The {\em order} of $G$ is denoted by $n=|V(G)|$. 
Two vertices $u$ and $v$ are {\em adjacent}, or {\em neighbors}, if $\{u,v\} \in E$. For $v \in V$, the {\em neighborhood} of $v$ is  $N(v) = \{u \in V: \{u,v\} \in E\}$  (or $N_G(v)$ if $G$ is needed for clarity), and the {\em closed neighborhood} of $v$ is the set $N[v] = N(v) \cup \{ v \}$. Similarly, for a set of vertices $S$, 
$N[S] = \cup_{v \in  S} N[v]$. The {\em degree} of a vertex $v$, denoted $\deg v$, is the cardinality of the set $N(v)$.

In a graph $G$, a vertex $v$ {\it dominates} all vertices in $N_G[v]$, a set of vertices $S$ {\it dominates} 
all vertices in $N_G[S]$, and in particular, when $N_G[S]=V(G)$ we say that $S$ is a {\it  dominating set} of $G$.  
A {\it minimum dominating set} is a dominating set of minimum cardinality, and its cardinality is the {\it domination number } of G, denoted $\gamma (G)$.

In \cite{HHHH}  the authors introduced the related concept of power domination by presenting propagation rules in terms of vertices and edges in a graph. Here we  use a simplified version of the propagation rules that is equivalent to the original (see   \cite{BH05}). For a  set $S\subseteq V$ in a graph $G=(V,E)$, define  $PD(S)\subseteq V$ recursively:

\ben
\item $PD(S)=N[S]$.
\item While there exists $ v\in PD(S)$ such that $|N(v)\setminus PD(S)|=1$: 
$PD(S)=PD(S)\cup N(v)$.
\een

A set $S\subseteq V(G)$ is a {\em power dominating set} of  a graph $G$ if $PD(S)=V(G)$ at the conclusion of the process above.   A {\em minimum power dominating set} is a power dominating set of minimum cardinality, and the {\em power domination number} of $G$,  denoted by $\pd (G)$,  is the cardinality of a minimum power dominating set.

We recall the following well-known upper bounds on the domination and power domination numbers of a graph that will be used throughout.
\begin{thm}[Ore's Theorem]\emph{\cite[Theorem 2.1]{DominationBook}}
\label{thm_n2_bound}
If $G$ is a graph on $n\geq 2$ vertices and $G$ has no isolated vertices, then $\gamma(G)\leq \frac{n}{2}$.
\end{thm}

\begin{thm}\emph{\cite{BH05,ZKC06}}
\label{thm_n3_bound}
If $G$ is a connected graph on $n\geq 3$ vertices, then $\gamma_P(G)\leq \frac{n}{3}$.
\end{thm}

Zero forcing can be described as a coloring game on the vertices of $G$.  The {\em color change rule} is: If $u$ is a blue vertex and exactly one  neighbor $w$ of $u$ is white, then change the color of $w$ to blue.  We say $u$ {\em forces} $w$ and denote this  by $u\to w$.
Given a set $B$ of blue vertices (all other vertices being white), the \emph{closure} of $B$, denoted $cl(B)$, is the set of blue vertices obtained after the color change rule is applied until no new vertex can be forced; it can be shown that $cl(B)$ is uniquely determined by $B$ (see \cite{AIM08}). A \emph{zero forcing set} is a set for which $cl(B)=V(G)$; a {\em minimum zero forcing set} is a zero forcing set of minimum cardinality, and the {\em zero forcing number} of $G$, denoted $Z(G)$, is the cardinality of a minimum zero forcing set.

A \emph{chronological list of forces}  $\mathcal{F}$ associated with a zero forcing set $B$ is a sequence of forces applied to obtain $cl(B)$ in the order they are applied. A \emph{forcing chain} for a chronological list of forces is a maximal sequence of vertices  $(v_1,\ldots,v_k)$ such that the force $v_i\to v_{i+1}$ is in $\mathcal{F}$ for $1\leq i\leq k-1$; we say the forcing chains are {\em associated with $B$}. Each forcing chain produces a distinct path in $G$, one of whose endpoints is in $B$; the other is called a \emph{terminal} (and the two can coincide).

\begin{defn}\label{restrictPD}
Let $G=(V,E)$ be a graph and let $X\subseteq V$.   A set $S\subseteq V(G)$ is a {\em power dominating set of   $G$ subject to $X$}  if $S$ is a power dominating set of   $G$ and $X\subseteq S$. The {\em restricted power domination number of $G$ subject to $X$} is the minimum number of vertices in a power dominating set that contains $X$, and is denoted by $\pd(G;X)$.

A set $S\subseteq V(G)$ is a {\em zero forcing set of   $G$ subject to $X$}  if $S$ is a zero forcing set of   $G$ and $X\subseteq S$.  The {\em restricted zero forcing number of $G$ subject to  $X$} is the minimum number of vertices in a zero forcing set that contains $X$, and is denoted by $Z(G;X)$.
\end{defn}


\begin{obs}  For any graph $G$, $\pd(G;\emptyset)=\pd(G)$ and $Z(G;\emptyset)=Z(G)$.
\end{obs}

\begin{ex}\label{ex:path}  For a path $P_n$ and any nonempty set $X$, $\pd(P_n;X)=|X|$, because any nonempty set of vertices is a power dominating set of a path. If  $X=\{x\}$ and $x$ is an endpoint of the path, then $Z(G;X)=1=Z(G)$; if $x$ is not an endpoint of the path, then $Z(G;X)=2$.
\end{ex}

The study of restricted power domination and restricted zero forcing fit into the general study of restricted domination parameters initiated by Goddard and Henning in  \cite{GH07} as a generalization of the work of   Sanchis     \cite{S97} and Henning \cite{H02} that investigated  the restricted version of standard graph domination.
One motivation for Sanchis' study of restricted domination  was a problem in resource allocation in distributed computer systems.  We follow the notation introduced by Henning (to avoid a conflict with an existing notation) and generalized by Goddard and Henning:   Let $G$ be a graph, and $X\subseteq V(G)$. Define $r(G,X,\gamma)$ to be the cardinality of a smallest dominating set of $G$ that contains $X$. The $k$-{\em restricted domination number} of $G$ is $r_k(G,\gamma)=\max \{ r(G,X,\gamma): X\subseteq V(G), |X|=k\}$. 

Returning to the more general setting and using the terminology of \cite{GH07}, let $\pi$ be a property of sets, such as being a dominating, power dominating, or zero forcing set and denote the minimum cardinality of a $\pi$-set by $f_\pi$.  Given a set of vertices $X$, define $r(G,X,f_\pi)$ to be the minimum cardinality of a $\pi$-set containing $X$.  Thus, $\pd(G;X)=r(G,X,\pd)$ and $Z(G;X)=r(G,X,Z)$.
Given a non-negative integer $k$, the {\em $k$-restricted $\pi$-number}  of  $G$ is $r_k(G,f_\pi)=\max \{ r(G,X,f_\pi): X\subseteq V(G), |X|=k\}$. 

\begin{rem}\label{rem:rd} Since for any graph $G$, $\gamma _P(G)\leq \gamma (G)$, it follows that  $r(G,X,\gamma_P)\leq  r(G,X,\gamma)$ for any $X\subseteq V(G)$.  As a consequence,  $r_k(G,\gamma_P)\leq r_k(G,\gamma)$ and $r(G,X,\gamma_P)\leq r_k(G,\gamma)$ for $k=|X|$.
\end{rem}

Other generalizations and extensions of power domination and zero forcing have also been explored; for example, variants of the problems obtained by modifying the color change rules have received significant attention \cite{CDMR12,chang15,FHKY17,yang17}, as has the problem of studying the number of timesteps involved in the process of zero forcing or power dominating a graph \cite{ferrero17,hogben12,liao16}.

Algorithmically, given a graph $G=(V,E)$ and sets $X\subseteq V$ and $S\subseteq V$, deciding if $S$ is a power dominating set of $G$ subject to $X$ is simply a matter of determining whether $S$ contains $X$ and whether $S$ is a power dominating set.  As a consequence, given an arbitrary graph $G=(V,E)$, an arbitrary set of vertices $X\subseteq V$, and an  integer $k\geq |X|$, the computational complexity of deciding if a graph $G$ has a power dominating set of cardinality $k$ is equivalent to that of determining if $G$ has a power dominating set of cardinality $k$ that also contains $X$. The former problem has been proven to be  $NP$-complete for bipartite graphs \cite[Theorem 5]{HHHH}, chordal graphs  \cite[Theorem 5]{HHHH}, circle and planar graphs \cite[Theorem 1]{GNR08}.


\section{Structural results and bounds}
\label{section_struct}

In this section, we present a number of structural results and bounds on the restricted power domination and restricted zero forcing numbers of a graph. We begin with the following simple bounds.

\begin{prop}\label{trivialbounds}
\label{obs_trivial}
For any graph $G=(V,E)$ and any set $X\subseteq V$, 
\begin{eqnarray*}
\gamma_P(G)&\leq& \gamma_P(G;X)\leq \gamma_P(G)+|X|,\\
Z(G)&\leq& Z(G;X)\leq Z(G)+|X|,\\
|X|&\le&\pd(G;X)\le Z(G;X),
\end{eqnarray*}
and all bounds are tight.
\end{prop}
\bpf  The lower bounds are all immediate from the definitions.  The first two upper bounds follow by adding the vertices of $X$ into a power dominating or zero forcing set, and the last follows from the fact that any zero forcing set is a power dominating set. For tightness: Let $X$ consist of $s$ leaves in $K_{1,p}$ where $1\le s\le p-2$, and note that $\pd(K_{1,p};X)=1+s=\pd(K_{1,p})+|X|$.   Various choices of $X$ in the path can be used to show the remaining bounds are tight; see Example \ref{ex:path}.
\epf

By a similar reasoning, the following more general bounds hold.
\begin{prop}
For any graph $G=(V,E)$ and any sets $Y\subseteq X\subseteq V$, 
\begin{eqnarray*}
\gamma_P(G;Y)&\leq& \gamma_P(G;X)\leq \gamma_P(G;Y)+|X\backslash Y|,\\
Z(G;Y)&\leq& Z(G;X)\leq Z(G;Y)+|X\backslash Y|,
\end{eqnarray*}
and all bounds are tight.
\end{prop}
In the next section, we give structural results related to the restricted power dominating 
sets of a graph. 
These lead to improved bounds on the restricted power domination and zero forcing numbers 
 in Sections \ref{ss:boundPD} and \ref{ss:boundZ}.    In Section \ref{ss:connex} we establish a relationship between 
 the restricted power domination number and the restricted zero forcing number.

\subsection{Structural results for power domination}\label{ss:structPD}

The next lemma states well-known results about mandatory vertices in power dominating sets; we include the brief proof for completeness (and because it has sometimes been misstated in the literature without proof).  A {\em leaf} is a vertex of degree one.

\begin{lem}
\label{obs_leaf2}
Let $G=(V,E)$ be a graph.
\ben[(1)]
\item If $v\in V$ is incident to three or more leaves, then $v$ is contained in every minimum power dominating set of $G$.
\item  If $v\in V$ is incident to exactly two leaves, then every minimum power dominating set contains either v or one of its two leaf neighbors.  
\item\label{cx3} There is a minimum power dominating set of $G$ that contains every vertex $v$ that is incident to exactly two leaves.
\een
\end{lem}
\bpf
First let $v\in V$ be incident to at least three leaves and suppose there is a minimum power dominating set $S$ of $G$ that does not contain $v$. If $S$ excludes two or more of the leaves of $G$ incident to $v$, then those leaves cannot be dominated or forced at any step. Thus, $S$ excludes at most one leaf incident to $v$, which means $S$ contains at least two leaves $\ell_1$ and $\ell_2$ incident to $v$. Then, $(S\backslash\{\ell_1,\ell_2\})\cup\{v\}$ is a smaller power dominating set than $S$, which is a contradiction.

Now consider the case in which $v\in V$ is incident to exactly two leaves, $\ell_1$ and $\ell_2$, and suppose there is a minimum power dominating set $S$ of $G$ such that $\{v,\ell_1,\ell_2\}\cap S = \emptyset$. Then neither $\ell_1$ nor $\ell_2$ can be dominated or forced at any step, contradicting the assumption that $S$ is a power dominating set.  If $S$ is a power dominating set that contains $\ell_1$ or $\ell_2$, say $\ell_1$, then $(S\backslash\{\ell_1\})\cup\{v\}$ is also a power dominating set and has the same cardinality.  Applying this to every vertex incident to exactly two leaves produces the minimum power dominating set required by \eqref{cx3}.
\epf

\begin{defn} \label{def_leaf}
Given a graph $G=(V,E)$ and a set $X\subseteq V$, define $\ell _r(G,X)$ as the graph obtained by attaching $r$ leaves to each vertex in $X$. If $X=\{v_1,\ldots,v_k\}$, we denote the $r$ leaves attached to vertex $v_i$ as $\ell _i^1,\ldots , \ell _i^r$ for each $i=1,\ldots ,k$, so that $V(\ell _r(G,X))=V\cup\{\ell_i^j:1\leq i\leq k, 1\leq j\leq r\}$ and $E(\ell _r(G,X))=E\cup\{v_i\ell_i^j:1\leq i\leq k, 1\leq j\leq r\}$.
\end{defn}

We will now establish a relationship between $\gamma_P(G;X)$ and $\ell_2(G,X)$ using Lemma \ref{obs_leaf2}.

\begin{prop}
	\label{prop_leafM}
	For any graph $G=(V,E)$ and any set $X\subseteq V$, $\gamma_P(G;X)=\gamma_P(\ell _2(G,X))$.
\end{prop}
\bpf
First note that any power dominating set of $G$ that contains $X$ is also a power dominating set of $\ell _2(G,X)$. Thus, $\gamma_P(\ell _2(G,X)) \leq \gamma_P(G;X)$. 

By Lemma \ref{obs_leaf2}, there exists a minimum power dominating set  $S$ of $\ell _2(G,X)$ that contains all vertices in $X$.  Clearly   $S$ is a  power dominating set of $G$ containing $X$, and this implies $\gamma_P(G;X) \leq \gamma_P(\ell _2(G,X))$. 
\epf

Since the motivation for studying the power domination number subject to a given set stems from the problem of optimizing PMU placement when the electric grid is expanded, we %
sometimes assume %
 that all graphs are connected and have at least three vertices  %
  (however, subgraphs of a given graph are allowed to be disconnected). This can make %
   the statements of some of  the following results cleaner, and they are easy to adapt to disconnected graphs that may have small components: Since $K_1$ and $K_2$ each have a power domination number of one, 
\[
\pd(G;X) = \sum_{\{i: |C_i|\geq3\}} \pd({C_i}; X\cap C_i) + \sum_{\{i:|C_i| \leq2\}} \max\{|X\cap C_i|, 1\}
\]
where  $C_1, \dots, C_r$ are the connected components of $G$.  

\begin{cor}
	\label{cor_leafboundM}
	Let $G$ be a connected graph of order $n\geq3$ and  $X \subseteq V(G)$. 
	Then $\gamma_P(G;X) \leq \left \lfloor {{n+2|X|} \over 3} \right \rfloor$ and this bound is tight.
\end{cor}

\bpf
By Proposition \ref{prop_leafM}, $\gamma_P(G;X)=\gamma_P(\ell _2(G,X))$. Since $|V(\ell _2(G,X))|=n+2|X|\geq3$, 
it follows from Theorem \ref{thm_n3_bound}  that $\gamma_P(\ell _2(G,X)) \leq \left \lfloor{{n+2|X|} \over 3} \right \rfloor$.

To see that this bound is tight, let $G=(V,E)$ and $X\subseteq V$ and construct $G' = (V',E')$ by attaching two leaves to each vertex in $V\backslash X$, that is, let $G' = \ell_2(G, V\backslash X)$. Then $\pd(G';X) = |V| = \left\lfloor\frac{|V'| + 2|X|}{3}\right\rfloor$.
\epf

  As noted in Remark \ref{rem:rd},  
$\pd (G;X)\leq r_{k}(G,\gamma)$ for $X\subseteq V(G)$ and $k=|X|$.  
Let $G$ be a connected graph of order $n$ and minimum degree at least $2$.
It was shown in {\cite{H02}} that    $r_{k}(G,\gamma)\leq {{2n+3k}\over 5}$  for every integer $k$ such that $1\leq k\leq n$.
As a consequence, for every nonempty set $X\subseteq V(G)$, $\gamma_P(G;X) \leq {{2n+3|X|}\over 5}.$
However,  Corollary \ref{cor_leafboundM} gives a better bound.

Lemma \ref{obs_leaf2} 
ensures that $X$ is contained in \emph{some} but not \emph{every} minimum power dominating set of $\ell _2(G,X)$.  For example, in  Figure \ref{leaf.in.minPDS} let  $G$ be the path with vertex set $\{3,4,5\}$, let $X=\{3,5\}$, and note that $\{1,5\}$ is a minimum power dominating set of $\ell_2(G,X)$ that does not contain $X$. 

\begin{figure}[ht!]
    \centering
    \begin{tikzpicture}{
    \begin{scope}
    \node[label={180:$1$}] (A) at (-2.5,-1) {};
    \node[label={180:$3$}] (B) at (-1.5,0) {};
    \node[label={180:$2$}] (C) at (-2.5,1) {};
    \node[label={90:$4$}] (D) at (0,0) {};
    \node[label={0:$7$}] (E) at (2.5,-1) {};
    \node[label={0:$5$}] (F) at (1.5,0) {};
    \node[label={0:$6$}] (G) at (2.5,1) {};
    \draw (A) -- (B) -- (C);\draw (B)--(D)--(F);\draw (E)--(F)--(G);
    \end{scope}}
    \end{tikzpicture}		\vspace{-1mm}
    \caption{A graph used in several examples.\label{leaf.in.minPDS}}
\end{figure}
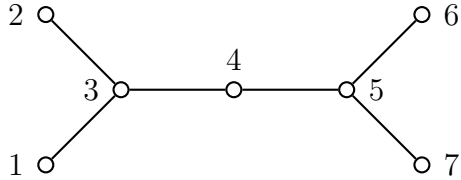

The next result shows that $X$ is contained in every minimum power dominating set of $\ell _3(G,X)$.   The proof is analogous to the proof of  Proposition \ref{prop_leafM} and is omitted.

\begin{prop}
\label{prop_leafB}
For any graph $G=(V,E)$ and any set $X\subseteq V$, a set $S$ is a minimum power dominating set of $G$ subject to $X$ if and only if $S$ is a minimum power dominating set of $\ell _3(G,X)$.
\end{prop}

\begin{lem}\label{leafnoforce} Let $G$ be any graph and let $u$ be  a leaf vertex of $G$. Then the following are equivalent:
\ben
\item $\gamma_P(G;\{u\})=\gamma_P(G-u)+1$.
\item There is some  minimum power dominating set of $G$ subject to $\{u\}$ in which $u$ does not need to perform a dominating step or a force.
\item There is some  minimum power dominating set of $G$ subject to $\{u\}$ of the form   $S\,\cup \{u\}$ where $S$ is a minimum power dominating set of $G-u$. 
\een
Furthermore, $\gamma_P(G;\{u\})=\gamma_P(G-u)$ if and only if $u$ needs to perform a dominating step or a force for any  minimum power dominating set of $G$ subject to $\{u\}$.  \end{lem}

\bpf   Since $u$ is a leaf,  a dominating step is the same as  a force for $u$. Furthermore, $\pd(G-u)\le\pd(G)\le\pd(G-u)+1$ and $\pd(G)\le \gamma_P(G;\{u\})\le \gamma_P(G-u)+1$.    
We show that $ (2) \Rightarrow (3) \Rightarrow (1) \Rightarrow (2) $.  Then the last statement follows from $(1) \Leftrightarrow (2)$.    

 If there is some  minimum power dominating set $S'$ of $G$ subject to $\{u\}$ in which $u$ does not need to perform a dominating step or a force, then $S=S'\setminus\{u\}$ is a power dominating set of $G-u$, and $S$ is minimum since $\pd(G)-1\le\pd(G-u)$.  

If there is some  minimum power dominating set of $G$ subject to $\{u\}$ of the form   $S\,\cup \{u\}$ where $S$ is a minimum power dominating set of $G-u$, then $\gamma_P(G;\{u\})=\gamma_P(G-u)+1$.

Now suppose  that $\gamma_P(G;\{u\})=\gamma_P(G-u)+1$. Let $w$ be the neighbor of $u$ in $G$. Then   $\gamma_P(G-u)+1=\gamma_P(G-uw)=\gamma_P(G-uw;\{u\})$.  Since $\gamma_P(G-uw;\{u\})=\gamma_P(G;\{u\})$,  the presence of the edge incident to $u$ does not affect the restricted power domination number, so $u$ does not need to perform a force in some minimum power dominating set $S'$ of $G$ subject to $\{u\}$. 
\epf

\subsection{Bounds for restricted power domination}\label{ss:boundPD}

If $W$ is a subset of the vertices of a graph $H$, the subgraph of $H$ induced by $W$ is denoted by $H[W]$.

\begin{prop}\label{propindsubg}
Let  $G'=(V',E')$ be a graph, $V\subset V'$, $G=G'[V]$, and $S$ be a power dominating set of $G$. Let $t$ be the number of isolated vertices of $G'[V'\backslash V]$.
   Then,
\[
\gamma_P(G';S)\leq |S|+\frac{|V'\backslash V|}{2}+\frac{t}{2}
\]
and this bound is tight.
\end{prop}

\bpf
Let $X_1$ be the set of isolated vertices of $G'[V'\backslash V]$. By Theorem \ref{thm_n2_bound}, the $|V'\backslash V|-t$ non-isolated vertices of $G'[V'\backslash V]$ can be dominated by a set $X_2$ of size at most $\frac{|V'\backslash V|-t}{2}$. Thus, the set $X_1\cup X_2$ can dominate all vertices in $G'[V'\backslash V]$ in the first time step, and then the rest of $G'$ can be power %
dominated by $S$. Thus, $\gamma_P(G';S)\leq |S|+\frac{|V'\backslash V|-t}{2}+t=|S|+\frac{|V'\backslash V|}{2}+\frac{t}{2}$.

To see that the bound is tight, consider $G=P_n \Box P_2$ with $n>3$, denote the vertices of $G$ by $(i,j)$ for $1 \leq i \leq n$ and $j=1,2$, and let $S=\{(1,1)\}$. Construct $G'$   from $G$ by adding one $K_2$ with both vertices adjacent to $(1,2)$ and another $K_2$ with both vertices adjacent to $(n,2)$  (see Figure \ref{fig:Prop3.8} for an example). Since for both $\ell=1$ and $\ell=n$, one of $(\ell,2)$ or a vertex from $N_{G'}[(\ell,2)]\setminus V(G)$ must be in any power dominating set of $G'$,  $\pd(G';S)=3 = 1 + \frac{4}{2} + \frac{0}{2}$.
\epf

\begin{figure}[ht!]
    \centering
    \begin{tikzpicture}{
    \begin{scope}
    \node (1A) at (-5.329,-.484) {};
    \node (2A) at (-5.224,0.707) {};
    \node[label={[label distance=-2pt]87:$(1,2)$}] (3A) at (-4,0) {};
    \node (4A) at (-2,0) {};
    \node (5A) at (0,0) {};
    \node (6A) at (2,0) {};
    \node[label={[label distance=-2pt]93:$(5,2)$}] (7A) at (4,0) {};
    \node (8A) at (5.329,-.484) {};
    \node (9A) at (5.224,.707) {};
    \node[label={[label distance=-2pt]235:$(1,1)$}] (3B) at (-4,-1.25) {};
    \node (4B) at (-2,-1.25) {};
    \node (5B) at (0,-1.25) {};
    \node (6B) at (2,-1.25) {};
    \node[label={[label distance=-2pt]315:$(5,1)$}] (7B) at (4,-1.25) {};
    \draw (1A)--(2A)--(3A)--(1A);
    \draw (3A)--(4A)--(5A)--(6A)--(7A)--(7B)--(6B)--(5B)--(4B)--(3B)--(3A);
    \draw (4A)--(4B);\draw (5A)--(5B);\draw (6A)--(6B);
    \draw (7A)--(8A)--(9A)--(7A);
    \end{scope}}
    \end{tikzpicture}		\vspace{-5mm}
\caption{Graph $G'$ for Proposition \ref{propindsubg}.\label{fig:Prop3.8}}
\end{figure}
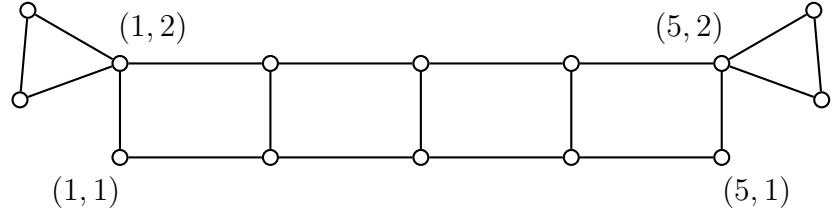

The next  proposition is a basic extension result. 

\begin{prop}\label{newlemLH}  Let  $G'=(V',E')$ be a graph, $V\subset V'$, $G=G'[V]$,   $H_1,\ldots,H_k$ be the components of $G'[V'\backslash V]$, and $S$ be a power dominating set of $G$.  For $1\leq i\leq k$, let $N_i=V(H_i)\cap N_{G'}[V\backslash N_G[S]]$. Then \[\gamma_P(G';S) \leq \sum_{i=1}^k\gamma_P(H_i;N_i)+|S|\]
 and this bound is tight.
\end{prop}
\bpf  For $i=1,\dots,k$, let  $ S_i$ be any power dominating set of $H_i$ that contains $N_i$.  We show $\cup_{i=1}^k S_i \cup S$ is a power dominating set of $G'$, so  choosing minimum power dominating sets $S_i$ subject to the restrictions establishes the inequality.  Suppose $y\in  N_{G'}[H_i]\cap V$, which implies $y$ has a neighbor $x\in V(H_i)$.   If $y\not\in N_G[S]$, then $x\in N_i$, so $S_i$ dominates $y$.  If $y\in N_G[S]$, then $y$ is dominated by $S$.  
Since  all neighbors of $H_i$ that are not in $H_i$
are dominated, all the vertices of $H_i$ can be power dominated by $S_i$.  Once all vertices not in $G$ are power dominated, $S$ can power dominate $G$.

To see that the bound is tight, let $G'$ be the graph in Figure \ref{leaf.in.minPDS}, $G=G'[\{1,2,3,4\}]$, and $S=\{3\}$.  Then $\pd(G';S)=2=1+1=\pd(G'[\{5,6,7\}];\{5\}) +|S|$. Larger examples may be constructed by subdividing the edge between 3 and 4 repeatedly. \epf

\begin{rem}\label{newlemBB}
The statement of Proposition \ref{newlemLH} remains valid when $N_i$ is re-defined as a subset of $V(H_i)$ which is a minimum dominating set of $V(H_i)\cap (N_{G'}[V]\backslash N_{G'}[S])$. The proof for this modified statement is similar to the proof of Proposition \ref{newlemLH}, except that the order is different: Since all the vertices in $N_{G'}[V]\setminus V$ are dominated, the vertices in $G$ can be power dominated, followed by the vertices of $H_i, i=1,\dots,k$. 
\end{rem}

\begin{prop}\label{prop310}
Let  $G'=(V',E')$ be a graph, $V\subset V'$, $G=G'[V]$,    and $S$ be a power dominating set of $G$. Suppose each component of $G'[V'\backslash V]$ has at least 3 vertices. Then,
\[
\gamma_P(G';S)\leq |S|+\frac{|V'\backslash V|}{3}+|N_{G'}[V\backslash N_G[S]]\cap (V'\backslash V)|
\]
and this bound is tight.
 \end{prop}
\bpf
  Let $H_1,\ldots,H_k$ be the components of $G'[V'\backslash V]$.  For $1\leq i\leq k$, let $N_i=V(H_i)\cap N_{G'}[V\backslash N_G[S]]$. From Proposition \ref{newlemLH},
$\gamma_P(G';S)\leq  \sum_{i=1}^k\gamma_P(H_i;N_i)+|S|$.
From Proposition \ref{trivialbounds}, $\gamma_P(H_i;N_i)\le \gamma_P(H_i)+|N_i|$ for $i=1,\dots,k$, and $\gamma_P(H_i)\le \frac{|V(H_i)|}{3}$ by Theorem \ref{thm_n3_bound}.  Thus, we conclude $\gamma_P(G';S)\leq \sum_{i=1}^k\frac{|V(H_i)|}{3} +\sum_{i=1}^k |N_i|+|S|$.
Since $\sum_{i=1}^k|V(H_i)|=|V'\backslash V|$ and  $\sum_{i=1}^k|N_i|=N_{G'}[V\backslash N_G[S]]\cap (V'\backslash V)$, we obtain $\gamma_P(G';S)\leq { {|V'\backslash V|}\over 3} +|N_{G'}[V\backslash N_G[S]]\cap(V'\backslash V)|+|S|.$

 To see that the bound is tight, consider $G=K_n$ with vertices $V=\{v_1,\ldots ,v_n\}$. For each $i=1,\ldots ,n$, add three vertices $a_i,b_i$ and $c_i$, and also add the edges $v_ia_i$, $a_ib_i$ and $a_ic_i$. Let $G'$ be the resulting graph so $|V'|=4n$ (see Figure \ref{fig:Prop3.10} for an example). 
  Then, $G=G'[V]$, $S=\{v_1\}$ is a power dominating set of $G$, and the connected components of  $G[V'\setminus V]$ are  the paths $b_i,a_i,c_i$ for $i=1,\ldots, n$. 
It is easily verified that $S\cup \{b_i:i=1,\ldots, n\}$ is a power dominating set and every $S$-restricted power dominating set must contain $S$ and at least one vertex in each component of $V'\setminus V$, so $\pd(G';S)= 1+n$.  Since $N_G[S]=V$, clearly $V\setminus N_G[S]=\emptyset$, so $|N_{G'}[V\setminus N_G[S]]\cap(V'\setminus V)|=0$. Since $|V'\setminus V|=3n$ and $|S|=1$, the upper bound is $|S|+{{|V'\setminus V|}\over 3}+|N_{G'}[V\setminus N_G[S]]\cap(V'\setminus V)|=1+n+0=\pd (G';S)$. 
\epf
 
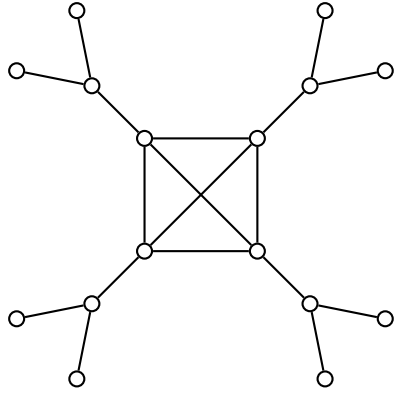
\begin{figure}[ht!]
    \centering
    \begin{tikzpicture}{
    \begin{scope}
	\node (A) at (-.75,-.75) {};
		\node (Aa) at (-1.45,-1.45) {};
		\node (Ab) at (-1.65,-2.45) {};
		\node (Ac) at (-2.45,-1.65) {};
		\draw (Ab)--(Aa)--(Ac);\draw (Aa)--(A);
	\node (B) at (-.75,.75) {};
		\node (Ba) at (-1.45,1.45) {};
		\node (Bb) at (-1.65,2.45) {};
		\node (Bc) at (-2.45,1.65) {};
		\draw (Bb)--(Ba)--(Bc);\draw (Ba)--(B);
	\node (C) at (.75,.75) {};
		\node (Ca) at (1.45,1.45) {};
		\node (Cb) at (1.65,2.45) {};
		\node (Cc) at (2.45,1.65) {};
		\draw (Cb)--(Ca)--(Cc);\draw (Ca)--(C);
	\node (D) at (.75,-.75) {};
		\node (Da) at (1.45,-1.45) {};
		\node (Db) at (1.65,-2.45) {};
		\node (Dc) at (2.45,-1.65) {};
		\draw (Db)--(Da)--(Dc);\draw (Da)--(D);
	\draw (A)--(B)--(C)--(D)--(A);\draw (A)--(C);\draw (B)--(D);
    \end{scope}}
    \end{tikzpicture}
\caption{Graphs $G=K_4$ and  $G'$ for Proposition \ref{prop310}.\label{fig:Prop3.10}}
\end{figure}

  \begin{prop}
\label{lemma_boundary}
Let $G=(V,E)$ be a graph, and $V=V_1\cup V_2$ be a partition of $V$. Let $N_1=N[V_2]\cap V_1$ and $N_2=N[V_1]\cap V_2$, and let $G_1=G[V_1]$ and $G_2=G[V_2]$. Let $W_1\subseteq V_1$ and $W_2\subseteq V_2$ be sets of vertices such that $W_1\cup W_2$ dominates $N_1\cup N_2$ in $G$. Then, 

\[
\gamma_P(G)\leq \gamma_P(G;W_1\cup W_2)\leq \gamma_P(G_1;W_1)+\gamma_P(G_2;W_2)
\]
and these bounds are tight.
\end{prop}
\bpf
Let $S_1$ and $S_2$ be sets that realize $\gamma_P(G_1;W_1)$ and $\gamma_P(G_2;W_2)$, respectively. By construction, $S_1\cup S_2$ contains $W_1\cup W_2$. Moreover, the vertices in $S_1\cup S_2$ can dominate $N_1\cup N_2$ in $G$ in the first time step, and then $S_1$ and  $S_2$, respectively, can force the vertices of $V_1$ and $V_2$ in $G$ independently, since the vertices in $V_1$ (respectively, $V_2$) are not adjacent to any uncolored vertices outside $V_1$ (respectively, $V_2$). Thus, $S_1\cup S_2$ is a power dominating set of $G$, so $\gamma_P(G;W_1\cup W_2)\leq \gamma_P(G_1;W_1)+\gamma_P(G_2;W_2)$. 
 
To see that the bound is tight, let $G_1=K_{1,p}$ and $G_2=K_{1,q}$ with $p,q\geq 4$. Label the centers of the stars $c_1$ and $c_2$, respectively. Let $u_1$ and $v_1$   be two leaves of $G_1$ and $u_2$ and $v_2$ be two leaves of $G_2$. Let $G$ be the graph formed from the disjoint union of $G_1$ and $G_2$ by adding all possible edges between a vertex in $\{u_1,v_1\}$ and a vertex in  $\{u_2,v_2\}$  (see Figure \ref{fig:Prop3.11} for an example). For the obvious partition of the vertices into $V(G_1), V(G_2)$, let $W_1=\{c_1\}$ and $W_2=\{c_2\}$. Then $\gamma_P(G)=2= \gamma_P(G;W_1\cup W_2)= \gamma_P(G[V_1];W_1)+\gamma_P(G[V_2];W_2)$.
\epf\vspace{-10pt}

\begin{figure}[ht!]
    \centering
    \begin{tikzpicture}{
    \begin{scope}
    \node (A) at (-3.75,-1) {};
    \node[label={180:$c_1$}] (B) at (-2.75,0) {};
    \node (C) at (-3.75,1) {};
    \node[label={270:$u_1$}] (D) at (-1,-.75) {};
    \node[label={90:$v_1$}] (E) at (-1,.75) {};
    \node[label={270:$u_2$}] (F) at (1,-.75) {};
    \node[label={90:$v_2$}] (G) at (1,.75) {};
    \node (H) at (3.75,-1) {};
    \node[label={0:$c_2$}] (I) at (2.75,0) {};
    \node (J) at (3.75,1) {};
    \draw (A) -- (B) -- (C);
    \draw (B)--(D)--(F)--(I)--(G)--(E)--(B);\draw (D)--(G);\draw (E)--(F);
    \draw (H)--(I)--(J);
    \end{scope}}
    \end{tikzpicture}		\vspace{-2mm}
\caption{Graph $G$ for Proposition \ref{lemma_boundary}.\label{fig:Prop3.11}}
\end{figure}
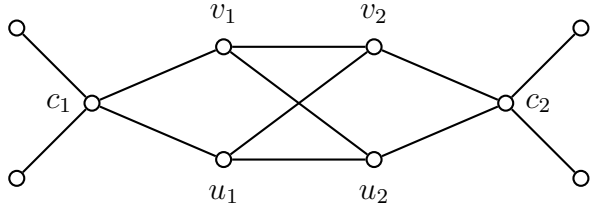

\subsection{Bounds for restricted zero forcing}\label{ss:boundZ}

The next two results are the zero forcing versions of Remark  \ref{newlemBB} and Proposition \ref{lemma_boundary}. 

 \begin{prop}
	\label{Zgeneral_boundary}  Let  $G'=(V',E')$ be a graph, $V\subset V'$, $G=G'[V]$,    $H_1,\ldots,H_k$ be the components of $G'[V'\backslash V]$, and $B$ be a zero forcing set of $G$. For $1 \leq i \leq k$, let $N_i=N_{G'}[V] \cap V(H_i)$. Then,
	\[Z(G';B) \leq \sum_{i=1}^kZ(H_i;N_i)+|B|\]
and this bound is tight.
\end{prop}
\bpf  For $i=1,\dots,k$, let  $ B_i$ be any zero forcing set of $H_i$ that contains $N_i$.  
We show $\cup_{i=1}^k B_i \cup B$ is a zero forcing set of $G'$, so choosing minimum zero forcing sets $B_i$ subject to the restrictions establishes the inequality. Since all vertices in $N_{G'}[V]\backslash V$ are initially colored blue, $B$ can force all vertices of $G$. After this, since all vertices in $N_{G'}[V(H_i)]\backslash V(H_i)$ are colored blue for each $H_i$, $B_i$ will be able to force $V(H_i)$ for $1\leq i\leq k$.

To see that the bound is tight, let $G'$ be the graph in Figure \ref{leaf.in.minPDS}, $G=G'[\{1,2,3\}]$, and $B=\{1\}$.  Then $H_1=\{4,5,6,7\}$, $N_1=\{4\}$, and $Z(G';B)=3=Z(H_1;N_1)+|B|$. 
\epf


\begin{prop}
	\label{lemma_Zboundary}
	Let $G=(V,E)$ be a graph, and $V=V_1\cup V_2$ be a partition of $V$. Let $N_1=N[V_2]\cap V_1$ and $N_2=N[V_1]\cap V_2$, and let $G_1=G[V_1]$ and $G_2=G[V_2]$. Then,
\[
 Z(G) \leq \min(Z(G_1)+Z(G_2;N_2), Z(G_1;N_1)+Z(G_2)),  
\]
and this bound is tight.
\end{prop}

\bpf
Without loss of generality, assume $Z(G_1)+Z(G_2;N_2)\le  Z(G_1;N_1)+Z(G_2)$. Start with an initial set $B=B_1 \cup B_2$ where $B_1$ and $B_2$ are minimum zero forcing sets of $G_1$ and $G_2$, respectively, and $N_2\subseteq B_2$. Given a vertex $v$ in $V_1$, its only white neighbors belong to $G_1$, hence $B_1$ can force all of $G_1$. Once $G_1$ is colored blue, $B_2$ is a zero forcing set of $G_2$ and will force the rest of the graph. 
To see that this bound is tight, consider the graph $G$ shown in Figure \ref{fig:Prop3.13}.  Partition the vertices as $V_1=\{1,2,3\}$ and $V_2=\{4,5,6,7,8\}$, so $N_1=\{2\}$ and $N_2=\{6,7,8\}$.  Then 
$Z(G)=3=2+1=Z(G_1;N_1)+Z(G_2)$.
\epf

\begin{figure}[ht!]
    \centering
    \begin{tikzpicture}{
    \begin{scope}
    \node[label={200:$4$}] (A) at (-4.5,0) {};
    \node[label={250:$8$}] (B) at (-2,0) {};
    \node[label={200:$7$}] (C) at (0,-1) {};
    \node[label={[label distance=2pt]175:$2$}] (D) at (0,1) {};
    \node[label={290:$6$}] (E) at (2,0) {};
    \node[label={340:$5$}] (F) at (4.5,0) {};
    \node[label={180:$1$}] (G) at (-.75,2.5) {};
    \node[label={0:$3$}] (H) at (.75,2.5) {};
    \draw (A)--(B)--(C)--(E)--(F);
    \draw (B)--(D)--(E);\draw (C)--(D);
    \draw (D)--(G)--(H)--(D);
    \end{scope}}
    \end{tikzpicture}		\vspace{-2mm}
\caption{Graph $G$ for Proposition \ref{lemma_Zboundary}.\label{fig:Prop3.13}}
\end{figure}
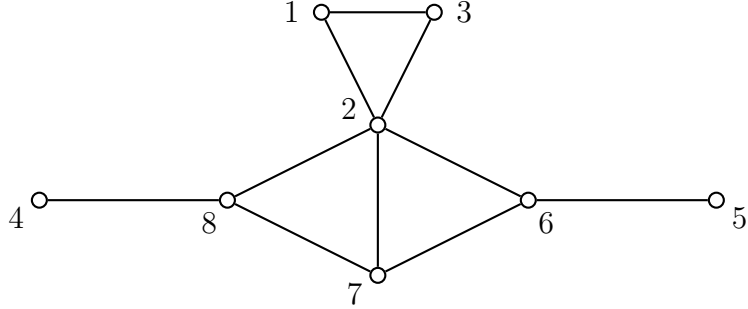


We conclude this section by considering $Z(G;X)$ when $X$ consists of a single vertex $v$.  To this end, we recall the concept of \emph{zero forcing spread} of a vertex, as defined in \cite{zspread}.  
Let $G$ be a graph and $v$ be a vertex in $G$. The \emph{zero forcing spread} of $v$ is defined as $z_v(G)=Z(G)-Z(G-v)$.

It is shown in \cite{zspread} that for any vertex $v$ in a graph $G$, $-1\leq z_v(G)\leq 1$,  and that $z_v(G)=1$ if and only if there exists a minimum zero forcing set of $G$ containing $v$ such that $v$ does not perform a force.  Furthermore, if $z_v(G)=-1$, then $v$ is not contained in any minimum zero forcing set of $G$. Since $Z(G)\leq Z(G; \{v\})\leq Z(G)+1,$ we have the following result. 
\begin{prop} For a graph $G$ and $v\in V(G)$
\[
Z(G;\{v\})=  \begin{cases} 
 Z(G) & \text{if } z_v(G)=1 \\
 Z(G)+1 & \text{if } z_v(G)=-1.  \\
 \end{cases}
\]
\end{prop}

For a graph $G$ and vertex $v\in V(G)$ with $z_v(G)=0$, it may be that $Z(G;\{v\})=Z(G)$ or $Z(G;\{v\})=Z(G)+1$. In Figure \ref{zvzero}, $z_v(G)=0$ and $Z(G;\{v\})=Z(G)=2$, while $z_u(G)=0$ and $Z(G;\{u\})  =3=Z(G)+1$.

\begin{figure}[ht!]
	\centering
	\begin{tikzpicture}{
		\begin{scope}
		\node (A) at (-2,0) {};
		\node (B) at (-1.25,-1) {};
		\node (D) at (0,-1) {};
		\node (E) at (1.75,0) {};
		\node (F) at (.75,0) {};
		\node [label={100:$u$}] (H) at (-0.65,1) {};
		\node [label={100:$v$}] (I) at (-3,0) {};
	    \node (J) at (-1.25,-2) {};
	    \node (K) at (0,-2) {};
	    
	  	\draw (A) -- (B);\draw (J)--(B)--(D)--(F);\draw (E)--(F);\draw (I)--(A)--(H)--(F);\draw (D)--(K);
	
		\end{scope}}
	\end{tikzpicture}
	\caption{Graph $G$ in which $z_v(G)=0$ and $Z(G;\{v\})=Z(G)$, while $z_u(G)=0$ and $Z(G;\{u\})=Z(G)+1$.}\label{zvzero}
\end{figure}
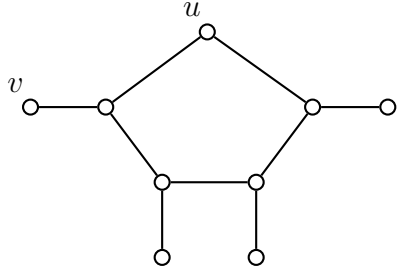

 Note that by numbering the desired vertex $v$ as the lowest number 0,  the Bruteforce algorithm implemented in the software \cite{mrsage} will return the desired vertex in a minimum zero forcing set if this is possible.  Thus the Bruteforce algorithm can be used to determine whether $ Z(G; \{v\})= Z(G)$ or $ Z(G; \{v\})= Z(G)+1$.

\subsection{Connections between restricted power domination and restricted zero forcing}\label{ss:connex}

Let the maximum degree of $G$ be denoted by $\Delta(G)$.

\begin{rem}
If $S$ is a power dominating set of $G$ that contains $X$, then $N[S]$ is a zero forcing set of $G$ that contains $N[X]$. Since $|N[S]|\leq (\Delta (G) +1)|S|$, we have  
\beq Z(G;N[X])\leq  (\Delta (G) +1)\gamma_P(G;X).\label{naivebd}\eeq
\end{rem}

The next results improve the bound in \eqref{naivebd}; their proofs use 
ideas from the proofs of \cite[Lemma 2]{Deanetal11} and \cite[Theorem 3.2]{REUF2015}, but differ in ensuring that the set $X$ is not inadvertently discarded.

\begin{lem}\label{lem:degbd}
	 Let $G$ be a graph with no isolated vertices and let $S=\{u_{1},\dots,u_t\}$ be a power dominating set  for $G$ that contains $X$.   
	 Then $Z(G;X)\le\sum_{i=1}^t\deg u_i$ and this bound is tight.
\end{lem}
\bpf  
For $i=1,\dots, t$, define
\[U_i=\left\{\begin{array}{ll}
N[u_i]\setminus \{v_i\} &  \mbox{  where $v_i\in N(u_i)\setminus S$ if $N(u_i)\setminus S\ne \emptyset$}\\
\{u_i\} &   \mbox{ if }N(u_i)\setminus S= \emptyset,\end{array}\right.
\] and $B=\cup_{i=1}^t U_i.$
Then $|B|\le \sum_{i=1}^t\deg u_i,$ because $|N[u_i]\setminus \{v_i\}|\le \deg u_i+1-1=\deg u_i$ when $N(u_i)\setminus S\ne \emptyset$, and when $N(u_i)\setminus S= \emptyset$,  $|\{u_i\}|=1$ and $\deg u_i\ge 1$.  

We show that $B$
is a zero forcing set of G.  Since $B\supseteq S\supseteq X$, this implies $Z(G;X)\le| B|\le \sum_{i=1}^t\deg u_i.$
Color  all the vertices in $B$ blue and the vertices in $V(G)\setminus B$ white. Then every vertex of $S$ is colored blue, and for $i=1,\dots, t$, at most one vertex in $N[u_i]$  is white  ($v_i$ is the only possibility).  So $u_i$ can immediately force $v_i$ if necessary, and $N[S]$ is colored blue.  Since $S$ is a power dominating set of $G$,
$N[S]$ is a zero
forcing set of G, and therefore $B$ is a zero forcing set of $G$.  

Whenever $Z(G)\le\sum_{i=1}^t\deg u_i$ is tight and  $X\subseteq S$, then the bound $Z(G;X)\le\sum_{i=1}^t\deg u_i$ is tight.  A path with $S=X$ consisting of an endpoint is an example.  \epf

\begin{thm}\label{betterbd}
Let $G$ be a graph with  $\Delta (G)\ge 1$.  Then $\lc\frac{Z(G;X)}{\Delta(G)}\rc\le \pd(G;X)$ and this bound is tight.
\end{thm}
\bpf  Suppose $G$ has connected components $G_1,\dots, G_h$.  Let $X_j=X\cap V(G_j)$ and suppose $G_i$ is a component   that has an edge, so $G_i$ does not have isolated vertices.   Choose a minimum power dominating set   $S_i=\{u_1^{(i)},\dots,u_{t_i}^{(i)}\}\supseteq X_i$  for $G_i$, so $t_i=\pd(G_i;X_i)$.     
Then by Lemma \ref{lem:degbd},  $Z(G_i;X_i)\le \sum_{i=1}^{t_i}\deg u_i^{(i)}\le t_i\Delta(G_i)=\pd(G_i;X_i)\Delta(G_i)$.
Thus  $\pd(G_i;X_i)\ge\lc\frac{Z(G_i;X_i)}{\Delta(G_i)}\rc\ge \lc\frac{Z(G_i;X_i)}{\Delta(G)}\rc$.  

Since  $\Delta(G)\ge 1$, we have $\pd(G_j;X_j)\ge \lc\frac{Z(G_j;X_j)}{\Delta(G)}\rc$ for every component $G_j$  of $G$ (including isolated vertices).  Thus, 
\[
\pd(G;X)=\sum_{j=1}^h\pd(G_j;X_j)\ge \sum_{j=1}^h\lc\frac{Z(G_j;X_j)}{\Delta(G)}\rc\ge\lc \frac{\sum_{j=1}^h Z(G_j;X_j)}{\Delta(G)}\rc=\lc\frac{Z(G;X)}{\Delta(G)}\rc. 
\]

Whenever $\lc\frac{Z(G)}{\Delta(G)}\rc\le \pd(G)$ is tight and  $X\subseteq S$, then $\lc\frac{Z(G;X)}{\Delta(G)}\rc\le \pd(G;X)$ is tight.    An example is $K_n$ because $Z(K_n)=\Delta(K_n)=n-1$ and $\pd(K_n)=1$.
\epf


\section{Exact and algorithmic results}
\label{sexactPD}

In this section, we present exact results and algorithms for $\gamma_P(G;X)$ and $Z(G;X)$ for certain families of graphs.

\subsection{Graphs with bounded treewidth}

Our first result uses Proposition \ref{prop_leafB} to show that the restricted power domination number of any graph with bounded treewidth can be computed efficiently.

\begin{thm}\label{computeTC}
For any graph $G=(V,E)$ with bounded treewidth and any set $X\subseteq V,$ a minimum power dominating set of $G$ subject to $X$ can be computed in $O(n)$ time.
\end{thm}
\bpf
The treewidth of $G$ does not change when an arbitrary number of leaves are appended to some or all of the vertices of $G$, so if $G$ has treewidth at most $k$ then $\ell _3(G,X)$ also has treewidth at most $k$. By Proposition \ref{prop_leafB}, a set $S$ is a minimum power dominating set of $G$ subject to $X$ if and only if $S$ is a minimum power dominating set of $\ell _3(G,X)$. Since  $|V(\ell _3(G,X))|\le 4|V(G)|$, the linear time algorithm presented by Guo et al.  in \cite{GNR08} applied to $\ell _3(G,X)$ yields a minimum power dominating set of $\ell _3(G,X)$ in $O(n)$ time.  Thus, $\gamma_P(G;X)$ can also be obtained in $O(n)$ time. 
\epf

Next we use some results of Section \ref{section_struct} to design a parallel algorithm for computing the power domination numbers of trees. 

 \begin{thm}
\label{prop_tree_parallel}
Let $T$ be a tree, $v$ be a vertex of $T$ with $\deg v\ge 2$, $V_1,\ldots,V_k$ be the vertex sets of the components of $T-v$, and $T_i=T[V_i\cup \{v\}]$ for $1\leq i\leq k$. Partition the indices $1,\ldots,k$ as follows:
\begin{eqnarray*}
&I&=\{i:\gamma_P(T_i;\{v\})=\gamma_P(T_i) \text{ and } 1+\gamma_P(T_i-v)\neq\gamma_P(T_i;\{v\})\},\\
&I'&=\{i:\gamma_P(T_i;\{v\})=\gamma_P(T_i) \text{ and } 1+\gamma_P(T_i-v)=\gamma_P(T_i;\{v\})\},\\
&J&=\{i:\gamma_P(T_i;\{v\})\neq\gamma_P(T_i)\}. 
\end{eqnarray*}
 Let $g=-k+\sum_{i=1}^k\gamma_P(T_i;\{v\})$. Then, 
\[
\gamma_P(T)=
\begin{cases}
g+1 &\text{ if }  |I|\geq 2 \text{ or } |J|=0\\
g &\text{ if } |I|\leq 1 \text{ and } |J|\geq 1.
\end{cases}
\]
\end{thm}

\bpf
Note that $v$ is a leaf in each $T_i$.  If $v$ is in a power dominating set of $T$, forcing can occur independently in each $T_i$ after $v$ and its neighbors have been dominated, so $\gamma_P(T;\{v\})=g+1$.  Since $\gamma_P(T) \leq \gamma_P(T;\{v\}) \leq \gamma_P(T)+1$, it follows that $g\leq \gamma_P(T)\leq g+1$. 

By Lemma \ref{leafnoforce}, $I$ is the index set of the subtrees of $T$ for which it is beneficial to include $v$ in a power dominating set, and where $v$ performs a dominating step or force; $I'$ is the index set of the subtrees of $T$ for which it is beneficial to include $v$ in a power dominating set, but where $v$ does not need to perform a dominating step or force; $J$ is the index set of the subtrees of $T$ for which it is not beneficial to include $v$ in a power dominating set.


For $1\leq i\leq k$, let $S_i$ be a set realizing $\gamma_P(T_i;\{v_i\})$ in $T_i$. Note that by definition of $J$, for $i\in J$, $S_i$ can be chosen such that $S_i\backslash\{v\}$ is a set realizing $\gamma_P(T_i)$. By Lemma \ref{leafnoforce}, for $i\in I'$, $S_i$ can be chosen such that $v$ does not need to perform a force. Suppose first that $|I|\leq 1$ and $|J|\geq 1$; we claim that $\bigcup_{i=1}^k(S_i\backslash\{v\})$ is a power dominating set of $T$. To see why, note that for each $i\in J$, the set $S_i\backslash\{v\}$ will force all of $V(T_i)$ in $T$, including $v$. Then for $i\in I'$, all components $T_i$ can be forced by the sets $S_i\backslash\{v\}$, $i\in I'$, since $v$ is colored but does not need to perform a force in those components.  Finally, if there is a component $T_{i^*}$, $i^*\in I$, $v$ will have a single uncolored neighbor at this step of the forcing process (which is in $T_{i^*}$), and it can force this neighbor; since $v$ is a leaf in $T_{i^*}$, this is the same as dominating its neighbor. Thus, $S_{i^*}\backslash\{v\}$ can power dominate $T_{i^*}$ after all other components are colored. Since $S_i\backslash\{v\}$, $1\leq i\leq k$ are pairwise disjoint, it follows that
\[
\left|\bigcup_{i=1}^k(S_i\backslash\{v\})\right|=\sum_{i=1}^k|S_i\backslash\{v\}|=\sum_{i=1}^k(\gamma_P(T_i;\{v\})-1)=g\geq \gamma_P(T)\geq g.
\]

Now suppose $|I|\geq 2$; then, $v$ is beneficial and forcing in at least 2 components, thus it must be contained in some minimum power dominating set of $T$. Similarly, if $|J|=0$ but $v$ is not in a minimum power dominating set of $T$, then none of the sets $S_i\backslash\{v\}$ can force $v$. In both cases, it follows that $\gamma_P(T)=g+1$. 
\epf

Note that by Theorem \ref{computeTC}, the parameter $g$ and the index sets $I$ and $J$ in Theorem \ref{prop_tree_parallel} can be determined by computing the ordinary power domination numbers of several smaller trees. If the vertex $v$ is picked in a way that separates $T$ into roughly equally-sized components, Theorem \ref{prop_tree_parallel} gives an efficient way to compute $\gamma_P(T)$ in parallel, since several processors can be used to compute the power domination numbers of the smaller trees independently.

%
%

\subsection{Integer programming formulations}

In this section, we present integer programming models for finding a minimum restricted power dominating set of a graph. Some of the following definitions and models are adapted from \cite{BFH17}. A \emph{fort} of a graph $G=(V,E)$ is a nonempty set $F\subset V$ such that no vertex outside $F$ is adjacent to exactly one vertex in $F$.

\begin{prop}
\label{prop_fort}
Let $G=(V,E)$ be a graph and $F$ be any fort of $G$.
\ben
\item\label{fpd} If $S$ is a power dominating set of $G$, then $S\cap N[F]\neq \emptyset$.
\item If $B$ is a zero forcing set of $G$, then $S\cap F\neq \emptyset$.
\een
\end{prop}
\bpf
We prove the first statement; the proof of the second is a simplified version of that proof. Suppose to the contrary  
that $S\cap N[F]=\emptyset$, and let $v$ be the first vertex in $F$ to be observed by $S$. Since no neighbor of $v$ is in $S$, $v$ is not observed in the domination step. Thus, $v$ must be forced by some vertex $w$ not in $F$. However, by definition of a fort, $w$ is adjacent to at least two vertices in $F$, so $w$ cannot force $v$; this is a contradiction.
\epf

We now present an integer program which can be used to find a restricted power dominating set subject to $X$. The binary variable $p_v$ indicates whether vertex $v$ is in the restricted power dominating set; $\mathcal{S}$ is the set of all forts of the given graph $G=(V,E)$.

\begin{Model}{IP model for restricted power domination based on forts}
\label{model_forts}
\begin{eqnarray}
\nonumber\min && \sum_{v \in V}p_{v}\\
\emph{s.t.:} && \sum_{v \in N[F]}p_v \geq 1  \qquad \forall F \in \mathcal{S}\label{M1C1}\\
&&p_v=1 \qquad\qquad\;\;\, \forall v \in X\label{M1C2}\\
\nonumber &&p_v\in\{0,1\}\qquad\quad\forall v \in V
\end{eqnarray}
\end{Model}
\begin{thm}
\label{thm4}
 The optimum of Model~\ref{model_forts} is equal to $\gamma_P(G;X)$.
\end{thm}
\bpf
Let $p$ be an optimal solution of Model~\ref{model_forts}, and let $S$ be the set of all vertices $v$ for which $p_v=1$. By constraint (\ref{M1C2}), $S$ contains $X$. Suppose to the contrary that $S$ is not a power dominating set of $G$, which means $cl(N[S]) \neq V$. If any vertex $u\in cl(N[S])$ is adjacent to exactly one vertex $v\in V\backslash cl(N[S])$, then $u$ could force $v$, contradicting the definition of $cl(N[S])$. Thus, $V\backslash cl(N[S])$ is a fort. Moreover, since $V\backslash cl(N[S])$ does not contain any vertices of $N[S]$, $N[V\backslash cl(N[S])]$ does not contain any vertex of $S$. This means constraint (\ref{M1C1}) is violated by the fort $V\backslash cl(N[S])$, a contradiction. It follows that $S$ is a restricted power dominating set of $G$ subject to $X$, and is minimum due to the objective function. 

Next, let $S$ be a minimum restricted power dominating set of $G$ subject to $X$, and $p$ be the vector whose nonzero entries are indexed by the vertices in $S$. By Proposition \ref{prop_fort}, for every $F\in \mathcal{S}$, $N[F]$ contains an element of $S$; thus, $p$ satisfies constraint (\ref{M1C1}). Moreover, $S$ contains $X$, so $p$ satisfies constraint (\ref{M1C2}). Thus, $p$ is a feasible solution of Model~\ref{model_forts}.
\epf

By an argument  similar to that given in the proof of Theorem \ref{thm4}, it follows that if constraint (\ref{M1C1}) of Model \ref{model_forts} is replaced by $\sum_{v\in F} p_v\geq 1 \;\forall F\in \mathcal{S}$, the optimum of the modified integer program will be equal to $Z(G;X)$.

Since a graph (e.g. $K_n$) could have an exponential number of forts, Model~\ref{model_forts} must in general be solved using constraint generation. This approach requires a practical method for finding violated constraints; to this end, we present an auxiliary integer program for generating violated fort constraints in Model~\ref{model_constraints}. In this model, $S$ is the set of all vertices for which $p_v = 1$ in the current optimal solution of Model~\ref{model_forts}; thus, $S$ is constant in Model~\ref{model_constraints}. The binary variable $f_v$ indicates whether vertex $v$ is in the fort. 

\begin{Model}{IP model for finding violated fort constraints}
\label{model_constraints}
\begin{eqnarray}
\nonumber\min&& \sum_{v \in V}f_{v}\\
\emph{s.t.:}&& \sum_{v \in V}f_v \geq 1\label{M3C1}\\
&&f_w +\sum_{u \in N(w)\backslash \{v\}} f_u \geq f_v\qquad\forall(v,w) \emph{ with } v \in V, w \in N(v)\label{M3C2}\\
&&f_v = 0\qquad\qquad\qquad\qquad\;\forall v \in cl(N[S])\label{M3C3}\\
\nonumber&& f_v\in\{0,1\} \qquad\qquad\qquad\;\;\forall v \in V
\end{eqnarray}
\end{Model}
\begin{thm}
Model~\ref{model_constraints} finds a minimum-size violated fort.
\end{thm}
\bpf
Let $f$ be a solution of Model~\ref{model_constraints}, and let $F$ be the set of vertices of $G$ for which $f_v=1$. By constraint (\ref{M3C1}), $F$ is not empty.  By constraint (\ref{M3C2}), every neighbor of a vertex in $F$ must either be in $F$, or have at least two neighbors in $F$. Thus, $F$ is a fort of $G$. Furthermore, by constraint (\ref{M3C3}), no vertex in $F$ is in $cl(N[S])$. Therefore $F$ is a violated fort with respect to the current solution $S$ of  Model~\ref{model_forts}, and is minimum due to the objective function.

Next, let $F$ be a violated fort of $G$; let $f_v=1$ for $v \in F$, and $f_v=0$ for $v\notin F$. Since a fort is nonempty by definition, $F$ must contain at least one vertex; therefore, constraint (\ref{M3C1}) is satisfied. Also by definition, any neighbor $w$ of a vertex $v$ in $F$ must either be in $F$ or have at least one other neighbor in $F$; therefore, constraint (\ref{M3C2}) is satisfied. Since $F$ is a violated fort, no vertex of $F$ can be in $cl(N[S])$; therefore, constraint (\ref{M3C3}) is satisfied. Thus, $f$ is a feasible solution to Model~\ref{model_constraints}.
\epf

\subsection{Graphs with polynomially-many terminals}

In this section, we identify certain conditions that allow the restricted zero forcing number of a graph to be computed efficiently, given a zero forcing set of a proper subgraph.

\begin{prop}
\label{prop_terminals_cond}
Let $X$ be a minimum zero forcing set of a graph $G=(V,E)$ and let $G_T$ be the graph obtained by adding a vertex $v^*$ to $G$ and connecting it to all vertices in a set $T\subset V$. 
\begin{enumerate}
\item If $T$ is a set of terminals of forcing chains associated with $X$, then $X$ is a zero forcing set of $G_T$, i.e., $Z(G_T;X)=Z(G)$.
\item If $X$ is a zero forcing set of $G_T$, then at least one vertex in $T$ is a terminal of a forcing chain associated with $X$ in $G$.
\end{enumerate}
\end{prop}
\bpf
If $T$ is a set of terminals of forcing chains associated with $X$ for some chronological list of forces, then any force performed in $G$ can also be performed in $G_T$ at the same step, since $v^*$ would not interfere with any forces performed between vertices of $G$. Then, $v^*$ could be forced by any of its neighbors in $G_T$ in the last step. Thus, $X$ is also a zero forcing set of $G_T$, so $|X|\geq Z(G_T;X)\geq |X|=Z(G)$.

Now suppose $X$ is a zero forcing set of $G_T$. Let $\mathcal{F}$  be a chronological list of forces for $X$ in $G_T$, and let $u\in T$ be the vertex which forces $v^*$. If $v^*$ does not perform a force, then $\mathcal{F}\backslash (u\rightarrow v^*)$ is a valid chronological list of forces for $X$ in $G$, and $u$ is a terminal of a forcing chain since it does not force any vertex in $G$. If $v^*$ does perform a force in $G_T$, say $v^*\rightarrow w$, let $\mathcal{F}'\subset \mathcal{F}$ be the subsequence of $\mathcal{F}$ beginning with the first force and ending with the force directly before the force $v^*\rightarrow w$. Then the set of vertices in $V(G)$ colored by $\mathcal{F}'$ is also a zero forcing set of $G$. Let $\mathcal{F}''$ be a chronological list of forces for this set in $G$, and note that $u$ does not perform a force in $\mathcal{F}''$, since if it did, it could not have forced $v^*$ in $G_T$; moreover, some other vertex in $V(G)$ forces $w$. Then, combining $\mathcal{F}'$ and $\mathcal{F}''$, we obtain a chronological list of forces for $G$ associated with $X$ in which  $u$ is the terminal of a forcing chain.
\epf

Note that the necessary conditions in Proposition \ref{prop_terminals_cond} are sometimes sufficient and the sufficient conditions are sometimes necessary. Note also that while the conditions assure that $X$ is a zero forcing set of $G_T$, they do not assure that $X$ is a \emph{minimum} zero forcing set of $G_T$. For example, let $G$ be the star $K_{1,4}$ and $T$ be a set containing two leaves of $G$; then $Z(G)=3$, and any zero forcing set of $G$ is a zero forcing set of $G_T$, but $G_T$ also has a zero forcing set of size 2.

\begin{prop}
\label{prop_terminals2}
Let $X$ be a minimum zero forcing set of $G$, $\{u_1,\ldots,u_k\}$ be a set of terminals of forcing chains associated with $X$, and $H_1,\ldots,H_k$ be disjoint connected graphs with $V(H_i)\cap V(G)=\{u_i\}$ for $1\leq i\leq k$. Let $G'=G\cup H_1\cup\cdots\cup H_k$. Then
\[Z(G';X)=|X|-k+\sum_{i=1}^k Z(H_i;\{u_i\}).\]
\end{prop}
\begin{proof}
For $1\leq i\leq k$, let $X_i$ be a minimum zero forcing of $H_i$ containing $u_i$. Then, the set $X\cup ((X_1\cup \cdots \cup X_k)\backslash \{u_1,\ldots,u_k\})$ is a zero forcing set of $G'$, since $X$ can force all vertices in $V(G)$ without any of the vertices in $V(G')\backslash V(G)$ interfering in this forcing process, and then each $H_i$ can be forced independently by $X_i$. Thus, since $X$ and each $X_i\backslash\{u_i\}$ are pairwise disjoint, and since $|X_i\backslash\{u_i\}|=Z(H_i;\{u_i\})-1$, it follows that $Z(G';X)\leq |X|-k+\sum_{i=1}^k Z(H_i;\{u_i\})$.

Now suppose  that  $B'$ is a minimum zero forcing set of $G'$  that contains $X$. 
If $B'$ contains some $u_i\notin X$, then the set $B'\backslash \{u_i\}$ is also a zero forcing set of $G'$, since it contains $X$ and can therefore force $u_i$ in some step of the forcing process. Thus, without loss of generality,  $B'$ excludes all $u_i\notin X$. 
For $i=1,\dots,k$, let  $X_i=B'\cap V(H_i)$.  If $u_i\not\in B'$, then $X'_i=X_i\cup\{u_i\}$ is a zero forcing set of $H_i$, since no vertex other than $u_i$ can be forced by a vertex outside $V(H_i)$.  Similarly, for each $u_i\in B'$, $X'_i=X_i$ is a zero forcing set of $H_i$. Notice that $u_i\in X'_i$ 
for $i=1,\dots,k$. Hence, $Z(G';X)=|B'|\geq |X\cup\, \bigcup_{i=1}^k \left(X'_i\backslash \{u_i\}\right)|=|X|+\sum_{i=1}^k (|X'_i|-1)\ge |X|-k+\sum_{i=1}^k Z(H_i;\{u_i\})$.
\end{proof}

In view of Proposition \ref{prop_terminals2}, given a graph with the structure of $G'$, with subgraph $G$ and a minimum zero forcing set $X$ of $G$, one could determine whether $Z(G';X)$ is as in the statement of the proposition in polynomial time, if there are polynomially-many sets of vertices which are terminals of forcing chains associated with $X$ (since then it could be checked whether $u_1,\ldots,u_k$ belong to one of these sets). Graphs with this property include paths, cycles, and complete graphs. However, the next result shows that in general, one could not efficiently enumerate all sets which are terminals of forcing chains associated with $X$. 

\begin{prop}
\label{prop_exp2}
A graph with a fixed minimum zero forcing set or minimum power dominating set $X$ can have exponentially-many sets of vertices that are terminals of forcing chains associated with $X$.
\end{prop}

\begin{proof}
Let $G'=(V,E)$ be the disjoint union of $k$ copies of $C_5$, where the $i^\text{th}$ copy of $C_5$ has vertex set $\{a_i,b_i,c_i,v_i,u_i\}$ and edge set $\{a_ib_i,b_ic_i,c_iv_i,v_iu_i,u_ia_i\}$. Let $G=(V\cup \{x\}, E\cup \{xv_i:1\leq i\leq k\}\cup \{xu_i:1\leq i\leq k\})$, and $X=\{x\}$. It is easy to see that $X$ is a minimum power dominating set of $G$. Initially $x$ dominates $\{u_i:1\leq i\leq k\}\cup\{v_i:1\leq i\leq k\}$. In each copy of $C_5$, possible forcing chains associated with $u_i$ and $v_i$ include $\{v_i\rightarrow c_i\rightarrow b_i\rightarrow a_i\}$, $\{u_i\rightarrow a_i,v_i\rightarrow c_i\rightarrow b_i\}$, $\{u_i\rightarrow a_i\rightarrow b_i,v_i\rightarrow c_i\}$, and $\{u_i\rightarrow a_i\rightarrow b_i\rightarrow c_i\}$, with sets of terminals $\{u_i,a_i\}$, $\{a_i,b_i\}$, $\{b_i,c_i\}$, and $\{c_i,v_i\}$, respectively. Thus, there are $\Omega(4^k)$ distinct sets of vertices that are terminals of forcing chains associated with $X$. 

Similarly, $X=\{a_1\}\cup\{u_i:1\leq i\leq k\}\cup\{v_i:3\leq i\leq k\}$ is a minimum zero forcing set of $G$ and there are $\Omega(4^{k-1})$ distinct sets of vertices which are terminals of forcing chains associated with $X$.
\end{proof}

Finally, we briefly discuss when the quantity characterized in Proposition \ref{prop_exp2} is  polynomial in the order of the graph. Trivially, if $Z(G)=O(1)$ or $Z(G)=n-O(1)$, then there are polynomially-many sets of vertices which are terminals of forcing chains associated with some fixed zero forcing set $X$. It could also happen that these cardinalities are bounded by a polynomial even when $Z(G)$ does not have an extremal value.


\section{Concluding remarks}\label{sconc}

In this paper, we have studied the restricted power domination and restricted zero forcing problems, which have practical applications in power network monitoring and control of quantum systems. In particular, we have derived tight bounds on the restricted power domination and restricted zero forcing numbers of graphs, and related them to their non-restricted analogues as well as to other graph parameters. We also presented exact and algorithmic results for computing the restricted power domination and restricted zero forcing numbers of several families of graphs, including graphs with bounded treewidth, and graphs whose subgraphs have zero forcing sets with polynomially many sets of terminals.  We used the notion of restricted power domination to develop a parallel algorithm for the power domination number of trees, whereby a tree is partitioned into smaller trees whose power dominating sets can be computed independently.

We also presented integer programming formulations for computing the restricted power domination numbers of arbitrary graphs; these models use a set cover framework with constraint generation, rather than a step-dependent formulation  that is sometimes used in the literature (see, e.g., \cite{fan_watson} and the bibliography therein). 

Future work can focus on deriving other bounds on the restricted power domination and zero forcing numbers, and relating them to other restricted domination parameters in the framework of the work of Sanchis, Goddard, and Henning \cite{GH07,S97,H02}. The proposed exact algorithms can also be extended to more general graphs. For instance, since the power domination number of block graphs can be computed efficiently \cite{Wang2016,Xu06}, it may be possible to extend Theorems \ref{computeTC} and \ref{prop_tree_parallel} to block graphs, i.e., to compute the restricted power domination number of a block graph in $O(n)$ time, and to compute the power domination number of a block graph in parallel using restricted power domination. Computational results for general graphs and refinements of the integer programming Models \ref{model_forts} and \ref{model_constraints} can also be pursued.

\section*{Acknowledgements}
This research began at the American Institute of Mathematics workshop {\em Zero forcing and its applications} with support from National Science Foundation, DMS-1128242. The work of BB is supported in part by the National Science Foundation, Grant No. 1450681.  The authors thank AIM and NSF.

\end{document}